\documentclass[12pt,letterpaper, leqno]{amsart}

\usepackage[latin1]{inputenc}
\usepackage[T1]{fontenc}
\usepackage{amsfonts}
\usepackage{amsmath}
\usepackage{amssymb}
\usepackage{eurosym}
\usepackage{mathrsfs}
\usepackage{palatino}
\usepackage{xcolor}

\newcommand{\R}{\mathbb{R}}

\newcommand{\Z}{\mathbb{Z}}

\numberwithin{equation}{section}

\swapnumbers
\theoremstyle{plain}
\newtheorem{thm}[equation]{Theorem}
\newtheorem{lem}[equation]{Lemma}

\newtheorem{cor}[equation]{Corollary}

\theoremstyle{definition}

\theoremstyle{remark}
\newtheorem{rem}[equation]{Remark}

\newcommand{\pair}[2]{\langle #1,#2 \rangle}
\newcommand{\ave}[1]{\langle #1 \rangle}

\title{Upper Bound for Multi-parameter Iterated Commutators}

\author{Laurent Dalenc}
\address{Laurent Dalenc, Universit\'e Paul Sabatier, France}
\email{laurent.dalenc@math.univ-toulouse.fr}
\thanks{}

\author{Yumeng Ou}
\address{Yumeng Ou, Department of Mathematics, Brown University, Providence, RI, USA}
\email{yumeng\_ou@brown.edu}
\thanks{The second author is supported by the NSF Grant DMS-0901139}

\begin{document}
\maketitle

\begin{abstract}
We show that the product BMO space can be characterized by iterated commutators of a large class of Calder\'on-Zygmund operators. This result follows from a new proof of boundedness of iterated commutators in terms of the BMO norm of their symbol functions, using Hyt\"onen's representation theorem of Calder\'on-Zygmund operators as averages of dyadic shifts. The proof introduces some new paraproducts which have BMO estimates.
\end{abstract}

\section{Introduction}

In \cite{LPPW} the product BMO space on $\R^{d_1}\otimes\cdots\otimes\R^{d_t}$ was characterized by the multi-parameter iterated commutators of Riesz transforms. This extended to the product setting the classical results of R. Coifman, R. Rochberg and G. Weiss \cite{CRW}, a characterization of classical BMO in terms of boundedness on $L^2(\R^d)$ of the commutator of a singular integral operator with a multiplication operator, which by duality also implies a weak factorization result of $H^1(\R^d)$.

In the multi-parameter setting, let $M_b$ be the operator of pointwise multiplication by $b\in BMO_{prod}(\R^{\vec{d}})$. Let $T_i$ be the Calder\'on-Zygmund operators on $\R^{d_i}$. One seeks to characterize product BMO in terms of commutators in the sense that
\[
\|b\|_{BMO_{prod}}\lesssim\|[\ldots[[M_b,T_1],T_2]\ldots,T_t]\|_{L^2\rightarrow L^2}\lesssim\|b\|_{BMO_{prod}}
\]
where the first and second inequality will be referred to as lower bound and upper bound, respectively. 

In the case of Hilbert transform, the above result in bi-parameter setting was proved by M. Lacey and S. Ferguson in \cite{LF}, where the upper bound was first shown by S. Ferguson and C. Sadosky \cite{FS}. M. Lacey and E. Tervilleger \cite{LT} then extended the result to the multi-parameter setting. The Riesz transform result was proved by M. Lacey, S. Petermichl, J. Pipher and B. Wick in \cite{LPPW}, where they obtained a more general upper bound result for any Calder\'on-Zygmund operators of convolution type with high degree of smoothness. Later on in \cite{LPPW2} they simplified the proof of the upper bound for Riesz transforms by means of dyadic shifts. Very recently, the first author and S. Petermichl \cite{DP} proved the lower bound for a larger class of Calder\'on-Zygmund operators satisfying certain criteria. 

In this paper, we prove the upper bound for any given collection of Calder\'on-Zygmund operators. As a corollary, we prove new characterizations of product BMO in terms of commutators of Calder\'on-Zygmund operators.

The main theorem of the paper is the following.

\begin{thm}\label{main}
Let $b\in BMO_{prod}(\R^{\vec{d}})$ and $(T_i)_{1\leq i\leq t}$ be a collection of Calder\'on-Zygmund operators, with each $T_i$ acting on parameter $i$ of $\R^{\vec{d}}=\R^{d_1}\otimes\cdots\otimes\R^{d_t}$. Then,
\[
\|[\ldots[[M_b,T_1],T_2]\ldots,T_t]\|_{L^2\rightarrow L^2}\leq C\|b\|_{BMO_{prod}}
\]
where $C$ depends only on $\vec{d}$ and $\prod_{i=1}^t\|T_i\|_{CZ}$.
\end{thm}

One of the interesting results implied directly by the theorem is that a perturbation of a collection of operators characterizing product BMO still characterizes product BMO. In other words, characterizing families such as the Riesz transforms are stable under small perturbations in the sense that the Calder\'on-Zygmund operator norm of the perturbation terms are small. We organize this observation into the following corollary. 

\begin{cor}\label{cor}
Let $(T_{i, s_i})_{1\leq i\leq t, 1\leq s_i\leq n_i}$ be a family of Calder\'on-Zygmund operators characterizing the space $BMO_{\text{prod}}(\R^{\vec{d}})$, that is, $\exists C_1, C_2>0$, such that
\[
C_1\|b\|_{BMO_{prod}}\leq \sup_{1\leq i\leq t, 1\leq s_i\leq n_i}\|[\ldots[[M_b,T_{1,s_1}],T_{2,s_2}]\ldots,T_{t,s_t}]\|_{L^2\rightarrow L^2}\leq C_2\|b\|_{BMO_{prod}}.
\]
Then, $\exists \epsilon>0$ such that for any family of Calder\'on-Zygmund operators $(T'_{i,s_i})_{1\leq i\leq t,1\leq s_i\leq n_i}$ satisfying $\|T'_{i,s_i}\|_{CZ}\leq \epsilon$, the family $(T_{i,s_i}+T'_{i,s_i})_{1\leq i\leq t, 1\leq s_i\leq n_i}$ still characterizes $BMO_{prod}(\R^{\vec{d}})$.
\end{cor}

In particular, since Calder\'on-Zygmund operators form a linear space, whose norm can be made arbitrarily small by multiplying a small constant, it means that once we have a collection of operators characterizing BMO, we automatically obtain infinitely many collections of operators which also characterize BMO. More specifically, let $(T_{i, s_i})_{1\leq i\leq t, 1\leq s_i\leq n_i}$ be a family as in the corollary above, for any arbitrary family of Calder\'on-Zygmund operators $(T'_{i,s_i})_{1\leq i\leq t,1\leq s_i\leq n_i}$, there exist $\epsilon_1,\ldots,\epsilon_t>0$ such that for any $0<c_i<\epsilon_i,\,1\leq i\leq t$, the family $(T_{i,s_i}+c_iT'_{i,s_i})_{1\leq i\leq t, 1\leq s_i\leq n_i}$ characterizes $BMO_{prod}(\R^{\vec{d}})$.

The main tool in the proof of the main theorem is the representation theorem by T. Hyt\"onen \cite{Hy}, which states that any Calder\'on-Zygmund operator can be represented as an average of dyadic shift operators with respect to a probabilistic measure on a collection of dyadic grids. While the earliest version of this theorem appeared in \cite{Hy2}, here we choose to apply a slightly different one given in \cite{Hy}. In our proof, we will reduce the problem to the upper bound for commutators with dyadic shifts. This is the first use of Hyt\"onen's representation theorem to commutator theory. The novelty of this approach to the upper bound is twofold. First, the commutators with dyadic shifts which have infinite complexity in our case, are carefully studied and effectively reduced to paraproducts and another class of bounded operators. In contrast to typical methods dealing with multi-parameter theory, this allows our argument to be iterated. Second, new paraproducts and a similar type of operators are introduced, and this is where the delicate estimates in product theory are required.

The paper is organized as follows. In Section 2, we recall several preliminary results on dyadic shifts, representation theorem and multi-parameter paraproducts. In Section 3, a full proof of the main theorem in its one-parameter case is introduced, while the proof of the main theorem in arbitrarily many parameters is presented in Section 4.

\section*{Acknowledgement}

The authors wish to thank Stefanie Petermichl and Jill Pipher for suggesting to both of them this beautiful subject for their thesis and having made this collaboration possible. We would also like to thank the anonymous referee for suggestions that have greatly improved the exposition of the article and for raising a question about an earlier proof of Theorem \ref{main}.

\section{Preliminaries}\label{prep}

We give some essential background for the proof of the main theorem.

\subsection{Dyadic shifts and representation theorem}

Recall that while the standard dyadic grid is defined as 
\[
\mathscr{D}^0:=\{2^{-k}([0,1)^d+m):k\in\Z,m\in\Z^d\},
\]
for any parameter $\omega=(\omega_j)_{j\in\Z}\in(\{0,1\}^d)^{\Z}$, one can define an associated shifted dyadic grid as
\[
\mathscr{D}^{\omega}:=\{I\dot+\omega:I\in\mathscr{D}^0\}
\]
where
\[
I\dot+\omega:=I+\sum_{j:2^{-j}<\ell(I)}2^{-j}\omega_j.
\]
For a fixed shifted grid $\mathscr{D}^\omega$ and $i,j\in\Z_{+}$, a dyadic shift operator $S_\omega^{ij}$ is defined to be bounded on $L^2$ with operator norm less than 1. Specifically,
\[
S_\omega^{ij}f:=\sum_{K\in\mathscr{D}^\omega}\sum_{\substack{I\in\mathscr{D}^\omega, I\subset K\\\ell(I)=2^{-i}\ell(K)}}\sum_{\substack{J\in\mathscr{D}^\omega, J\subset K\\\ell(J)=2^{-j}\ell(K)}}a_{IJK}\pair{f}{h_I}h_J=:\sum_K\sum_{I,J\subset K}^{(i,j)}a_{IJK}\pair{f}{h_I}h_J,
\]
with $|a_{IJK}|\leq |I|^{1/2}|J|^{1/2}/|K|$. $S_\omega^{ij}$ is called cancellative if all the Haar functions in the definition are cancellative, otherwise, it is called noncancellative. 

Recall that in one dimension, any dyadic interval $I$ is associated with a cancellative Haar function $h_I^0=|I|^{-1/2}(\chi_{I_l}-\chi_{I_r})$ and a noncancellative one $h_I^1=|I|^{-1/2}\chi_I$. While in $d$ dimensions, each cube $I=I_1\times\cdots\times I_d$ is associated with $2^d$ Haar functions:
\[
h_I^\epsilon(x)=h_{I_1\times\cdots\times I_d}^{(\epsilon_1,\ldots,\epsilon_d)}(x_1,\ldots,x_d)=\prod_{i=1}^d h_{I_i}^{\epsilon_i}(x_i),\,\epsilon\in\{0,1\}^d,
\]
where $h_I^1$ is called noncancellative, while all the other $2^d-1$ Haar functions $h_I^\epsilon$ for $\epsilon\in\{0,1\}^d\setminus\{1\}$ are cancellative. Note that all the cancellative Haar functions for a fixed grid form an orthonormal basis of $L^2(\R^d)$. And in this paper, we usually suppress the parameter $\epsilon$ to abbreviate the notation.

We now introduce T. Hyt\"onen's representation theorem, a key tool in our proof. Interested readers can find its proof and a more detailed discussion in \cite{Hy} and \cite{Hy2}. The operator $T$ mentioned in the following will denote a Calder\'on-Zygmund operator associated with a $\delta$-standard kernel $K$. T. Hyt\"onen \cite{Hy} proved the following theorem:

\begin{thm}\label{repre}
Let $T$ be a Calder\'on-Zygmund operator, then it has an expansion, say for $f,g\in C^\infty_0(\R^d)$,
\[
\pair{g}{Tf}=c\cdot\|T\|_{CZ}\cdot\mathbb{E}_\omega\sum^\infty_{i,j=0}2^{-\max{(i,j)}\delta/2}\pair{g}{S^{ij}_\omega f},
\]
where $c$ is a dimensional constant and $S^{ij}_\omega$ is a dyadic shift of parameter $(i,j)$ on the dyadic grid $\mathcal{D}^\omega$; all of them except possibly $S^{00}_\omega$ are cancellative.
\end{thm}

According to the proof of Theorem \ref{repre}, in the representation of any $T$, only $S_\omega^{00}$ may be noncancellative, and if this is the case, only one of $\{h_I\}, \{h_J\}$ in its definition is noncancellative, i.e. $S_\omega^{00}$ is a paraproduct with some BMO symbol $a$ satisfying $\|a\|_{BMO}\leq 1$ and $a_I=\pair{a}{h_I}|I|^{-1/2}$, $\forall I\in\mathcal{D}$.

\subsection{Multi-parameter paraproducts}

Recall that a multi-parameter paraproduct associated with function $b$ can be viewed as a bilinear operator which is defined as
\[
B_0(b,f)=\sum_{R\in\mathscr{D}_{\vec{d}}} \beta_R\pair{b}{h_R^{\epsilon_1}}\pair{f}{h_R^{\epsilon_2}}h_R^{\epsilon_3}|R|^{-1/2},
\]
where $\epsilon_j\in\{0,1\}^{\vec{d}}$, $\mathscr{D}_{\vec{d}}$ denotes the tensor product of dyadic grids, and $\{\beta_R\}_R$ is a sequence satisfying $|\beta_R|\leq1$. Note that $h_R^{\epsilon_j}$ is cancellative if and only if $\epsilon_j\neq\vec{1}$. According to Journ\'e \cite{Jo} and later on improved by C. Muscalu, J. Pipher, T. Tao and C. Thiele \cite{MPTT} \cite{MPTT2}, one has the following boundedness result.

\begin{thm}\label{para}
Let $\vec{d}=(d_1,\ldots,d_t)$ and $\epsilon_j=(\epsilon_{j,1},\ldots,\epsilon_{j,t})$. If $\epsilon_1\neq\vec{1}$ and $\forall 1\leq s\leq t$, there is at most one of $j=2,3$ such that $\epsilon_{j,s}=\vec{1}$, then the operator $B_0$ satisfies
\[
B_0:\,BMO_{\text{prod}}(\R^{\vec{d}})\times L^2(\R^{\vec{d}})\rightarrow L^2(\R^{\vec{d}}).
\]
\end{thm}

\section{Proof of the one-parameter case}

In this section, we present a detailed proof of the main theorem in the one-parameter setting, which will later on be utilized to prove the multi-parameter result in the next section. As an essential part of the proof, delicate estimates of new paraproducts and a new operator $P$ will be introduced.
 
Given a BMO function $b$ and a Calder\'on-Zygmund operator $T$, one could represent the commutator $[b,T]$ as an average of $[b,S^{ij}_\omega]$ due to Theorem \ref{repre}. Then, in order to prove the upper bound inequality, it suffices to prove that for any $f\in C^\infty_0(\R^d)$,
\begin{equation}\label{bound}
\|\sum^\infty_{i,j=0}2^{-\max{(i,j)}\delta/2}[b,S^{ij}_\omega]f\|_{L^2}\lesssim \|b\|_{BMO}\|f\|_{L^2}
\end{equation}
uniformly in $\omega$. In the following we will write $S^{ij}$ for short as the argument doesn't depend on $\omega$ explicitly. 

As a crucial ingredient in our argument, two kinds of paraproduct-like operators need to be introduced.

The first one is the bilinear operator $B_k$ which could be viewed as a generalized dyadic paraproduct:
\[
B_k(b,f):=\sum_I \beta_I\pair{b}{h_{I^{(k)}}}\pair{f}{h_I}h_I|I^{(k)}|^{-1/2},
\]
where $\{\beta_I\}_I$ is a sequence satisfying $|\beta_I|\leq1$, $k\geq 0$ is an arbitrary integer, and $I^{(k)}$ denotes the $k$-th dyadic ancestor of $I$. Note that when $k=0$, this is exactly the classical paraproduct that we have introduced at the end of the previous section, whose boundedness is stated in Theorem \ref{para}. Lemma \ref{paraB1} below shows that such boundedness holds uniformly for any $B_k$. 

The second one is the trilinear operator $P$ defined as
\[
P(b,a,f):=\sum_I\pair{b}{h_I}\pair{f}{h_I}|I|^{-1}\sum_{J: J\subsetneq I}\pair{a}{h_J}h_J,
\]
which will be proved to be bounded on $BMO\times BMO\times L^2\rightarrow L^2$ in Lemma \ref{oneparaP}.

The main theorem we will prove in this section is the following:

\begin{thm}\label{decomp}
For cancellative dyadic shift $S^{ij}$, $[b,S^{ij}]f$ can be represented as a finite linear combination of the following terms:
\begin{equation}\label{term}
S^{ij}(B_k(b,f)),\quad B_k(b,S^{ij}f)
\end{equation}
where the integer $k$ is such that $0\leq k\leq\max(i,j)$ and the total number of terms is bounded by $C(1+\max(i,j))$ for some universal dimensional constant $C$. 

For noncancellative dyadic shift $S^{00}$ (dyadic paraproduct) with symbol $a$, $[b,S^{00}]f$ can be represented as a finite linear combination of the following terms:
\begin{equation}\label{term1}
S^{00}(B_0(b,f)),\quad B_0(b,S^{00}f),\quad P(b,a,f),\quad P^*(b,a,f),
\end{equation}
where $P^*$ is understood as the adjoint of $P$ with $b$ and $a$ fixed, and the total number of terms is bounded by a universal dimensional constant.
\end{thm}

\begin{rem}
The representation claimed in Theorem \ref{decomp} is far from unique. In fact, suggested by its proof, the readers can easily come up with representations of $[b,S^{ij}]f$ using other types of paraproducts, by decomposing the Haar sums differently. Moreover, as shown in the proof, the representation can be made such that except when $k=0$, all the Haar functions appearing in $B_k(b,f)$ are cancellative. 
\end{rem}

It is easy to see that Theorem \ref{main} is implied by Theorem \ref{decomp}. Indeed, given the boundedness of $S^{ij}$, Lemma \ref{paraB1}, Lemma \ref{oneparaP} together with Theorem \ref{para} guarantee the uniform boundedness of each of  the terms in (\ref{term}) and (\ref{term1}). Hence, 
\[
\|[b,S^{ij}]f\|_{L^2}\lesssim (1+\max(i,j))\|b\|_{BMO}\|f\|_{L^2}.
\]
Note that the uniform boundedness of $B_k$ with respect to $k$ is key in the above argument, which is also the main difficulty of the proof of Lemma \ref{paraB1}. Then, with the decaying factor $2^{-\max(i,j)\delta/2}$ in front, (\ref{bound}) follows from a simple geometric series argument.

\begin{lem}\label{paraB1}
Given $b\in BMO(\R^d)$ and $k\geq 0$, let
\[
B_k(b,f)=\sum_I \beta_I\pair{b}{h_{I^{(k)}}}\pair{f}{h_I}h_I|I^{(k)}|^{-1/2},
\]
where all the Haar functions are cancellative. Then $\|B_k(b,f)\|_{L^2}\lesssim \|b\|_{BMO}\|f\|_{L^2}$ with a constant independent of $k$.
\end{lem}

Before we proceed to its proof, note that for the application to our problem, there is no need to include cases when some of the Haar functions in $B_k$ are noncancellative according to the remark above. Hence, $B_k(b,f)$ is in fact a martingale transform whose uniform boundedness follows directly from the observation $|\pair{b}{h_{I^{(k)}}}|/|I^{(k)}|^{1/2}\leq\|b\|_{BMO}$. However, we will present a different proof via square function in the following, which will provide some insight into the estimates of some other operators and the multi-parameter analogs of the result, where noncancellative Haar functions have to be taken into account.
\begin{proof}
For any $g\in L^2(\R^d)$, 
\[
\pair{B_k(b,f)}{g}=\pair{b}{\sum_I\beta_I\pair{f}{h_I}\pair{g}{h_I}h_{I^{(k)}}|I^{(k)}|^{-1/2}}.
\]

It thus suffices to show that
\[
\|\sum_I\beta_I\pair{f}{h_I}\pair{g}{h_I}h_{I^{(k)}}|I^{(k)}|^{-1/2}\|_{H^1}\lesssim\|f\|_{L^2}\|g\|_{L^2},
\]
which is equivalent to
\[
\|S\big(\sum_I \beta_I\pair{f}{h_I}\pair{g}{h_I}h_{I^{(k)}}|I^{(k)}|^{-1/2}\big)\|_{L^1}\lesssim \|f\|_{L^2}\|g\|_{L^2},
\]
where in the above $S$ denotes the dyadic square function.

To see this, write
\[
S\big(\sum_I \beta_I\pair{f}{h_I}\pair{g}{h_I}h_{I^{(k)}}|I^{(k)}|^{-1/2}\big)^2=\sum_J\left(\sum_{I: I^{(k)}=J}\beta_I\pair{f}{h_I}\pair{g}{h_I}|J|^{-1/2}\right)^2\frac{\chi_J}{|J|}
\]
which together with $\|\cdot\|_{\ell^2}\leq \|\cdot\|_{\ell^1}$ and Cauchy-Schwarz inequality implies
\[
\begin{split}
&S\big(\sum_I \beta_I\pair{f}{h_I}\pair{g}{h_I}h_{I^{(k)}}|I^{(k)}|^{-1/2}\big)\\
&\leq \sum_J\left(\sum_{I: I^{(k)}=J}|\pair{f}{h_I}||\pair{g}{h_I}|\frac{\chi_J}{|J|}\right)\\
&\leq\sum_J\left(\sum_{I: I^{(k)}=J}|\pair{f}{h_I}|^2\right)^{1/2}\left(\sum_{I: I^{(k)}=J}|\pair{g}{h_I}|^2\right)^{1/2}\frac{\chi_J}{|J|}\\
&\leq\left(\sum_J\sum_{I: I^{(k)}=J}|\pair{f}{h_I}|^2\frac{\chi_J}{|J|}\right)^{1/2}\left(\sum_J\sum_{I: I^{(k)}=J}|\pair{g}{h_I}|^2\frac{\chi_J}{|J|}\right)^{1/2}\\
&=:(S^{(k)}f)(S^{(k)}g).
\end{split}
\]
where the operator $S^{(k)}f:=(\sum_J\sum_{I: I^{(k)}=J}|\pair{f}{h_I}|^2|J|^{-1}\chi_{J})^{1/2}$. We claim that $S^{(k)}: L^2\rightarrow L^2$ with norm bounded by a dimensional constant, which does not depend on $k$. This guarantees that our estimate of $B_k$ becomes independent of $k$. Combining this with another use of Cauchy-Schwarz  will complete the proof.

To show the claim, denote $\alpha_J=(\sum_{I: I^{(k)}=J}|\pair{f}{h_I}|^2)^{1/2}$ for any $J$ and define $F(x)=\sum_J\alpha_Jh_J(x)$. Then
\[
\begin{split}
\|S^{(k)}f\|_{L^2}^2&=\|(\sum_J\alpha_J^2\frac{\chi_J}{|J|})^{1/2}\|_{L^2}^2=\|SF\|_{L^2}^2\\
&\lesssim\|F\|_{L^2}^2=\sum_J\alpha_J^2=\sum_J\sum_{I: I^{(k)}=J}|\pair{f}{h_I}|^2=\sum_I|\pair{f}{h_I}|^2=\|f\|_{L^2}^2,
\end{split}
\]
where the second to last equality holds because that cube $I$ in the previous summation ranges over all the dyadic cubes exactly once. 
\end{proof}

\begin{lem}\label{oneparaP}
For tri-linear operator 
\[
P(b,a,f):=\sum_I\pair{b}{h_I}\pair{f}{h_I}|I|^{-1}\sum_{J: J\subsetneq I}\pair{a}{h_J}h_J,
\]
there holds
\[
\|P(b,a,f)\|_{L^2}\lesssim \|b\|_{BMO}\|a\|_{BMO}\|f\|_{L^2}.
\]
\end{lem}

\begin{proof}
The idea of the proof is to employ the $H^1$-BMO duality and the square function characterization of $H^1$. For any normalized test function $g\in L^2$,
\[
\pair{P(b,a,f)}{g}=\pair{b}{\sum_{I}\pair{f}{h_I}|I|^{-1}h_I\sum_{J: J\subsetneq I}\pair{a}{h_J}\pair{g}{h_J}}.
\]

To see where the BMO norm of $a$ comes into play, observe that for any fixed $I$ and some $1<p<2$,
\[
\begin{split}
&|\sum_{J: J\subsetneq I}\pair{a}{h_J}\pair{g}{h_J}|=|\pair{\sum_{J: J\subsetneq I}\pair{a}{h_J}h_J}{g\chi_I}|\\
\leq&\|\sum_{J: J\subsetneq I}\pair{a}{h_J}h_J\|_{L^{p'}}\|g\chi_I\|_{L^p}\lesssim\|(\sum_{J: J\subsetneq I}|\pair{a}{h_J}|^2\frac{\chi_J}{|J|})^{1/2}\|_{L^{p'}}\|g\chi_I\|_{L^p}\\
\lesssim&\|a\|_{BMO}|I|^{1/p'}\|g\chi_I\|_{L^p}=\|a\|_{BMO}|I|(\ave{|g|^p}_I)^{1/p},
\end{split}
\]
where the last inequality follows from John-Nirenberg inequality.

Therefore,
\[
\begin{split}
&S(\sum_{I}\pair{f}{h_I}|I|^{-1}h_I\sum_{J: J\subsetneq I}\pair{a}{h_J}\pair{g}{h_J})=\big(\sum_{I}|\pair{f}{h_I}|^2|I|^{-2}(\sum_{J: J\subsetneq I}\pair{a}{h_J}\pair{g}{h_J})^2\frac{\chi_I}{|I|}\big)^{1/2}\\
\leq&\|a\|_{BMO}\big(\sum_{I}|\pair{f}{h_I}|^2(\ave{|g|^p}_I)^{2/p}\frac{\chi_I}{|I|}\big)^{1/2}\leq\|a\|_{BMO}\big(\sum_I|\pair{f}{h_I}|^2\sup_{I: x\in I}(\ave{|g|^p}_I)^{2/p}\frac{\chi_I}{|I|}\big)^{1/2}\\
\leq&\|a\|_{BMO}M(|g|^p)^{1/p}S(f),
\end{split}
\]
where $M$ is the Hardy-Littlewood maximal function which is bounded on $L^p$, $1<p<\infty$. Hence,
\[
\|P(b,a,f)\|_{L^2}\lesssim\|b\|_{BMO}\|a\|_{BMO}\|M(|g|^p)^{1/p}\|_{L^2}\|S(f)\|_{L^2}\lesssim\|b\|_{BMO}\|a\|_{BMO}\|f\|_{L^2}.
\]
\end{proof}

Now we turn to the proof of Theorem \ref{decomp} and the strategy is the following. First, we decompose $b$ and $f$ using Haar bases. Second, we split the sum into several parts and represent each of them as a linear combination of terms in Theorem \ref{decomp}. 

To start with, one decomposes $[b,S^{ij}]f$ as
\[
\begin{split}
[b,S^{ij}]f&=\sum_{I,J}\pair{b}{h_I}\pair{f}{h_J}[h_I,S^{ij}]h_J\\
&=\sum_{I,J}\pair{b}{h_I}\pair{f}{h_J}\left(h_IS^{ij}h_J-S^{ij}(h_Ih_J)\right)=:I+II,
\end{split}
\]
where in the following $I$ and $II$ will be referred to as first term and second term, respectively. In order to further organize the sum and extract the correct paraproduct structure, even in the simplest one-parameter case, one needs to divide up the sum into many different parts, depending on the relative sizes of $I,J$.

\subsection{Cancellative dyadic shift $S^{ij}$}
Let's first look at the case when $S^{ij}$ is cancellative, meaning that all the Haar functions appearing are cancellative. Hence,
\[
[b,S^{ij}]f=\sum_{I,J}\pair{b}{h_I}\pair{f}{h_J}\left(h_I\sum_{J'\subset J^{(i)}}^{(j)}a_{JJ'J^{(i)}}h_{J'}-\sum_{K}\sum_{I'',J''\subset K}^{(i,j)}a_{I''J''K}\pair{h_Ih_J}{h_{I''}}h_{J''}\right).
\] 

First, we claim that it suffices to consider the part $I\subset J^{(i)}$. Indeed, it is obvious that when $I\cap J^{(i)}=\emptyset$, both terms in the parentheses are zero. Furthermore, by the cancellation structure of the commutator, when $I\supsetneq J^{(i)}$, the term $[h_I,S^{ij}]h_J$ is also zero. To see this, as $h_I$ is constant on $J^{(i)}$, fixing an arbitrary $x_0\in J^{(i)}$ implies
\[
h_IS^{ij}h_J-S^{ij}(h_Ih_J)=h_I(x_0)S^{ij}h_J-S^{ij}(h_I(x_0)h_J)=0.
\]
Note that for the case $(i,j)\neq (0,0)$, this is the only part of the proof where one needs the particular cancellation of the commutator structure.

Next, we represent the first term and the second term separately.

\subsubsection{First term}

Based on the discussion above, for any $i,j$, the first term containing $h_IS^{ij}h_J$ is equal to
\[
\sum_{J}\sum_{I: I\subset J^{(i)}}\pair{b}{h_I}\pair{f}{h_J}h_I\sum_{\substack{J': J'\subset J^{(i)}\\ \ell(J')=2^{i-j}\ell(J)}}a_{JJ'J^{(i)}}h_{J'}.
\]

Introducing index $K=J^{(i)}$ allows us to rewrite this as
\[
\begin{split}
&\sum_K\sum_{J: J\subset K}^{(i)}\sum_{I: I\subset K}\pair{b}{h_I}\pair{f}{h_J}h_I\sum_{J': J'\subset K}^{(j)}a_{JJ'K}h_{J'}\\
&=\sum_I\pair{b}{h_I}h_I\left(\sum_{K: K\supset I}\sum_{J: J\subset K}^{(i)}\sum_{J': J'\subset K}^{(j)}a_{JJ'K}\pair{f}{h_J}h_{J'}\right).
\end{split}
\]

Comparing the inner parentheses to the definition of $S^{ij}$ suggests that the expression above is equal to
\[
\begin{split}
&\sum_{I}\pair{b}{h_I}h_I\sum_{J': J'^{(j)}\supset I}\pair{S^{ij}f}{h_{J'}}h_{J'}\\
&=\sum_I\sum_{J': J'\supsetneq I}\pair{b}{h_I}\pair{S^{ij}f}{h_{J'}}h_Ih_{J'}+\sum_I\sum_{J': J'\subset I\subset J'^{(j)}}\pair{b}{h_I}\pair{S^{ij}f}{h_{J'}}h_Ih_{J'}=:I+II.
\end{split}
\]

Note that there are only part $I$ and $II$ left because of the supports of Haar functions. For part $I$, one writes
\[
\begin{split}
&I=\sum_I\pair{b}{h_I}h_I\left(\sum_{J': J'\supsetneq I}\pair{S^{ij}f}{h_{J'}}h_{J'}\right)=\sum_I\pair{b}{h_I}h_I\pair{S^{ij}f}{h_I^1}h_I^1\\
&=\sum_I\pair{b}{h_I}\pair{S^{ij}f}{h_I^1}h_I |I|^{-1/2},
\end{split}
\]
which is of type $B_0(b,S^{ij}f)$. In order to deal with part $II$, observe that it can be decomposed into finitely many pieces depending on the relative sizes of $I$ and $J'$, i.e.
\[
\begin{split}
&II=\sum_{k=0}^{j}\sum_{J'}\pair{b}{h_{J'^{(k)}}}\pair{S^{ij}f}{h_{J'}}h_{J'^{(k)}}h_{J'}\\
&=\sum_{k=0}^{j}\sum_{J'}\beta_{J'}\pair{b}{h_{J'^{(k)}}}\pair{S^{ij}f}{h_{J'}}h_{J'}|J'^{(k)}|^{-1/2}=\sum_{k=0}^{j}B_k(b,S^{ij}f),
\end{split}
\]
where $\beta_{J'}\in\{1,-1\}$ and $0\leq k\leq j$. Note that the sum at the end contains only $1+j\leq 1+\max(i,j)$ terms. Therefore, the representation of the first term is demonstrated.

\subsubsection{Second term}
Now we turn to the second term that contains $S^{ij}(h_Ih_J)$. Due to the supports of Haar functions, this part is nontrivial only when $I\cap J\neq\emptyset$. Hence, one can split this term into three parts: $I\subsetneq J$, $I=J$, and $J\subsetneq I\subset J^{(i)}$.

For $I\subsetneq J$, note that the second term becomes
\[
\begin{split}
S^{ij}(\sum_{I\subsetneq J}\pair{b}{h_I}\pair{f}{h_J}h_Ih_J)&=S^{ij}(\sum_I\pair{b}{h_I}h_I\sum_{J: J\supsetneq I}\pair{f}{h_J}h_J)\\
&=S^{ij}(\sum_I\pair{b}{h_I}h_I\pair{f}{h_I^1}h_I^1)\\
&=S^{ij}(\sum_I\pair{b}{h_I}\pair{f}{h_I^1}h_I|I|^{-1/2}),
\end{split}
\]
which is $S^{ij}(B_0(b,f))$. 

As the diagonal part $I=J$ is obviously of the form $S^{ij}(B_0(b,f))$ already, we move on to the last piece $J\subsetneq I\subset J^{(i)}$, which can be written as
\[
S^{ij}(\sum_J\sum_{I: J\subsetneq I\subset J^{(i)}}\pair{b}{h_I}\pair{f}{h_J}h_Ih_J).
\]

Observe that what's inside the parentheses is of an almost identical form as part $II$ that appeared at the end of the discussion of the first term except that $j$ is changed to $i$ and that $f$ takes the place of $S^{ij}f$. Hence, the same reasoning implies that it is a sum of at most $i\leq \max(i,j)$ terms of $S^{ij}(B_k(b,f))$, $1\leq k\leq i$. This proves the representation of the second term as well as completes the discussion of the case when $S^{ij}$ is cancellative.

\subsection{Noncancellative dyadic shift $S^{00}$}\label{Sec3}

It suffices to assume that 
\[
S^{00}f=\sum_I a_I\pair{f}{h_I^1}h_I,
\]
where $a_I:=\pair{a}{h_I}|I|^{-1/2}$ with $\|a\|_{BMO}\leq 1$. Because if we switch the positions of cancellative and noncancellative Haar functions, what we obtain is none other than its adjoint. Moreover, for the Haar expansion
\[
[b,S^{00}]f=\sum_{I,J}\pair{b}{h_I}\pair{f}{h_J}[h_I,S^{00}]h_J,
\]
it is not hard to see, according to a discussion similar to the one at the beginning of the case $(i,j)\neq (0,0)$, that one needs only to consider the part $I\subset J$ thanks to the commutator structure. We then split the sum into two parts: $I\subsetneq J$ and $I=J$.

\subsubsection{Part $I\subsetneq J$}

To decompose this part, once again we consider the first term containing $h_IS^{00}h_J$ and the second term containing $S^{00}(h_Ih_J)$ separately, without need to exploit more of the cancellation of the commutator. The second term can be dealt with exactly the same as how we treated the $I\subsetneq J$ part of the second term in the case $(i,j)\neq (0,0)$, which we omit. To study the first term, one observes that for any $h_J$, 
\[
S^{00}h_J=\sum_{I\subsetneq J}a_I\pair{h_J}{h_I^1}h_I=\sum_{I\subsetneq J}a_I|I|^{1/2}h_Ih_J.
\]

Hence, the first term becomes
\[
\sum_J\sum_{I, I'\subsetneq J}\pair{b}{h_I}h_I\pair{f}{h_J}a_{I'}|I'|^{1/2}h_{I'}h_J=\sum_J\sum_{I\subset I'\subsetneq J}+\sum_J\sum_{I'\subsetneq I\subsetneq J}=:I+II.
\]

One writes
\[
\begin{split}
I&=\sum_{I}\pair{b}{h_I}h_I\left(\sum_{I': I\subset I'}\sum_{J: I'\subsetneq J}a_{I'}\pair{f}{h_J}h_J|I'|^{1/2}h_{I'}\right)\\
&=\sum_I\pair{b}{h_I}h_I\left(\sum_{I': I\subset I'}a_{I'}|I'|^{1/2}h_{I'}\pair{f}{h_{I'}^1}h_{I'}^1\right)\\
&=\sum_I\pair{b}{h_I}h_I\left(\sum_{I': I\subset I'}a_{I'}\pair{f}{h_{I'}^1}h_{I'}\right)\\
&=\sum_I\pair{b}{h_I}h_I\left(\sum_{I': I\subset I'}\pair{S^{00}f}{h_{I'}}h_{I'}\right)\\
&=\sum_{I}\pair{b}{h_I}h_I\pair{S^{00}f}{h_I}h_I+\sum_I\pair{b}{h_I}h_I\pair{S^{00}f}{h_I^1}h_I^1\\
&=\sum_I\beta_I\pair{b}{h_I}\pair{S^{00}f}{h_I}h_I^\epsilon|I|^{-1/2}+\sum_I\pair{b}{h_I}\pair{S^{00}f}{h_I^1}h_I|I|^{-1/2},
\end{split}
\]
which is the sum of two $B_0(b,S^{00}f)$ with $\beta_I\in\{1,-1\}$.

To deal with part $II$, observe that
\[
II=\sum_{I'\subsetneq I}\pair{b}{h_I}h_Ia_{I'}|I'|^{1/2}h_{I'}\pair{f}{h_I^1}h_I^1,
\]
by first summing over index $J$. Thus,
\[
\begin{split}
II&=\sum_{I'}a_{I'}|I'|^{1/2}h_{I'}\left(\sum_{I: I\supsetneq I'}\pair{b}{h_I}|I|^{-1/2}\pair{f}{h_I^1}h_I\right)\\
&=:\sum_{I'}a_{I'}|I'|^{1/2}h_{I'}\sum_{I: I\supsetneq I'}\pair{S_bf}{h_I}h_I\\
&=\sum_{I'}a_{I'}\pair{S_bf}{h_{I'}^1}h_{I'}=S^{00}(S_bf),
\end{split}
\]
where the operator $S_bf:=\sum_I\pair{b}{h_I}|I|^{-1/2}\pair{f}{h_I^1}h_I$ is a classical paraproduct $B_0(b,f)$, and this completes the discussion of part $I\subsetneq J$.

\subsubsection{Part $I=J$}

In this special case, what we try to decompose becomes
\begin{equation}\label{special}
\sum_I\sum_{\epsilon,\epsilon'\in\{0,1\}^d\setminus\{\vec{1}\}}\pair{b}{h_I^\epsilon}\pair{f}{h_I^{\epsilon'}}\left(h_I^\epsilon S^{00}h_I^{\epsilon'}-S^{00}(h_I^\epsilon h_I^{\epsilon'})\right).
\end{equation}
Here, in order to avoid possible confusion, we wrote out the sum over index $\epsilon,\epsilon'$ explicitly. Recall that for each cube $I$, there are $2^d$ different Haar functions associated: $\{h_I^\epsilon\}$, $\epsilon\in\{0,1\}^d$, and the Haar function is noncancellative if and only if $\epsilon=\vec{1}$. First, it is useful to observe that if $\epsilon\neq\epsilon'$, $[h_{I}^{\epsilon},S^{00}]h_{I}^{\epsilon'}=0$. Indeed, for any fixed $I$ and $\epsilon,\epsilon'$,
\[
h_I^{\epsilon}S^{00}h_I^{\epsilon'}=\sum_{J: J\subsetneq I} a_J|J|^{1/2}h_J(h_I^\epsilon h_I^{\epsilon'}),
\]
and
\[
S^{00}(h_I^\epsilon h_I^{\epsilon'})=\sum_{J: J\supset I}a_J|J|^{-1/2}h_J\left(\int_I h_I^\epsilon h_I^{\epsilon'}\right)+\sum_{J: J\subsetneq I} a_J|J|^{1/2}h_J(h_I^\epsilon h_I^{\epsilon'}).
\]

As a result of cancellation and the fact that $\int_I h_I^\epsilon h_I^{\epsilon'}$ is nonzero if and only if $\epsilon=\epsilon'$, i.e. $h_I^\epsilon h_I^{\epsilon'}=|I|^{-1}\chi_I$, $[h_{I}^{\epsilon},S^{00}]h_{I}^{\epsilon'}\neq 0$ only when $\epsilon=\epsilon'$. Therefore, one can safely suppress the dependence on $\epsilon$ when studying this part of the sum. 

Furthermore, it is easily seen that the second term containing $S^{00}(h_Ih_I)$ here can be estimated exactly the same as before, it thus suffices to deal with the first term containing $h_IS^{00}h_I$, which is equal to
\[
\sum_I\pair{b}{h_I}\pair{f}{h_I}h_IS^{00}h_I=\sum_{I}\pair{b}{h_I}\pair{f}{h_I}|I|^{-1}\sum_{J: J\subsetneq I}\pair{a}{h_J}h_J=P(b,a,f),
\]
hence the proof is complete.

\section{Proof of the main theorem}

In this section, we present the proof of the main theorem in the general setting by iterating the one-parameter result, i.e. Theorem \ref{decomp}, in the previous section. For the sake of brevity, we consider the bi-parameter case as an example, while the strategy can be easily generalized to work for arbitrarily many parameters. The main idea is to show that the commutator can be represented as a finite linear combination of the bi-parameter analogs of terms in Theorem \ref{decomp}, for which one needs to define and estimate the following new bi-parameter operators, including all the possible "tensor products" of the one-parameter operators $B_k$ and $P$.

\begin{lem}\label{bipara}
Given $b\in BMO_{prod}(\R^n\times\R^m)$ and integers $k,l\geq 0$, define the following operators
\[
B_{k,l}(b,f)=\sum_{I_1,I_2}\beta_{I_1I_2}\pair{b}{h_{I_1^{(k)}}\otimes u_{I_2^{(l)}}}\pair{f}{h_{I_1}^{\epsilon_1}\otimes u_{I_2}^{\epsilon_2}}h_{I_1}^{\epsilon_1'}\otimes u_{I_2}^{\epsilon_2'}|I_1^{(k)}|^{-1/2}|I_2^{(l)}|^{-1/2},
\]
where $\beta_{I_1I_2}$ is a sequence satisfying $|\beta_{I_1I_2}|\leq 1$. When $k>0$, all the Haar functions in the first variable are cancellative, while when $k=0$, there is at most one of $h_{I_1}^{\epsilon_1}, h_{I_1}^{\epsilon_1'}$ being noncancellative. The same assumption goes for the second variable. Then, $\|B_{k,l}(b,f)\|_{L^2}\lesssim\|b\|_{BMO_{prod}}\|f\|_{L^2}$ with a constant independent of $k,l$.
\end{lem}

In the above, we use $u_{I_2}$ to denote Haar functions in the second variable, for any dyadic cube $I_2\subset \mathbb{R}^m$. Note that when $k=l=0$, $B_{k,l}$ becomes the classical bi-parameter $B_0$ we have seen at the end of Section \ref{prep}. When all the Haar functions are cancellative, the proof of the lemma proceeds exactly the same as its one-parameter counterpart, except that one needs bi-parameter dyadic square function as majorization instead. Therefore in the following, we will only prove the lemma assuming that $k=0, l>0$, and $h_{I_1}^{\epsilon_1}=h_{I_1}^1$ is the only noncancellative Haar. Note that in the setting of arbitrarily many parameters, parallel results still hold.

\begin{proof}
We are going to follow the strategy in the proof of Lemma \ref{paraB1} and use hybrid maximal-square functions as majorization.

Pairing $B_{0,l}(b,f)$ with a normalized $L^2$ function $g$ and applying the product $H^1$-$BMO$ duality, it suffices to show that
\[
\|SS(\sum_{I_1,I_2}\beta_{I_1I_2}\pair{f}{h_{I_1}^1\otimes u_{I_2}}\pair{g}{h_{I_1}\otimes u_{I_2}}h_{I_1}\otimes u_{I_2^{(l)}}|I_1|^{-1/2}|I_2^{(l)}|^{-1/2})\|_{L^1}\lesssim \|f\|_{L^2},
\]
where $SS$ is the dyadic double square function whose $L^1$ norm characterizes product $H^1$.

To see this, one calculates
\[
\begin{split}
&SS(\sum_{I_1,I_2}\beta_{I_1I_2}\pair{f}{h_{I_1}^1\otimes u_{I_2}}\pair{g}{h_{I_1}\otimes u_{I_2}}h_{I_1}\otimes u_{I_2^{(l)}}|I_1|^{-1/2}|I_2^{(l)}|^{-1/2})^2\\
&=\sum_{I_1,I_2}\left(\sum_{J_2: J_2^{(l)}=I_2}\pair{f}{h_{I_1}^1\otimes u_{J_2}}\pair{g}{h_{I_1}\otimes u_{J_2}}|I_1|^{-1/2}|I_2|^{-1/2}\right)^2\frac{\chi_{I_1}\otimes\chi_{I_2}}{|I_1||I_2|}\\
&\leq \sum_{I_1}\left(\sum_{I_2}\sum_{J_2: J_2^{(l)}=I_2}\sup_{I_1}(\ave{\pair{f}{u_{J_2}}_2}_{I_1})\pair{g}{h_{I_1}\otimes u_{J_2}}\frac{\chi_{I_2}}{|I_2|}\right)^2\frac{\chi_{I_1}}{|I_1|},
\end{split}
\]
where the last inequality follows from $\|\cdot\|_{\ell^2}\leq\|\cdot\|_{\ell^1}$, and $\ave{\cdot}_{I_1}$ denotes the average value over $I_1$. Then the above is controlled by
\[
\sum_{I_1}\left(\sum_{I_2}\sum_{J_2: J_2^{(l)}=I_2}M_1(\pair{f}{u_{J_2}}_2)|\pair{g}{h_{I_1}\otimes u_{J_2}}|\frac{\chi_{I_2}}{|I_2|}\right)^2\frac{\chi_{I_1}}{|I_1|},
\]
where $M_1$ is the Hardy-Littlewood maximal function in the first variable. Next, Cauchy-Schwarz inequality implies that
\[
\begin{split}
&\leq \sum_{I_1}\left(\sum_{I_2}\sum_{J_2: J_2^{(l)}=I_2}M_1(\pair{f}{u_{J_2}}_2)^2\frac{\chi_{I_2}}{|I_2|}\right)\left(\sum_{I_2}\sum_{J_2: J_2^{(l)}=I_2}|\pair{g}{h_{I_1}\otimes u_{J_2}}|^2\frac{\chi_{I_2}}{|I_2|}\right)\frac{\chi_{I_1}}{|I_1|}\\
&=\left(\sum_{I_2}\sum_{J_2: J_2^{(l)}=I_2}M_1(\pair{f}{u_{J_2}}_2)^2\frac{\chi_{I_2}}{|I_2|}\right)\left(\sum_{I_1}\sum_{I_2}\sum_{J_2: J_2^{(l)}=I_2}|\pair{g}{h_{I_1}\otimes u_{J_2}}|^2\frac{\chi_{I_1}\otimes\chi_{I_2}}{|I_1||I_2|}\right)=:I\cdot II.
\end{split}
\]

$II$ could be written as the square of $SS$ acting on a normalized $L^2$ function, similarly as the last part of the proof of Lemma \ref{paraB1}. For $I$, Fefferman-Stein inequality implies that
\[
\begin{split}
\|I^{1/2}\|_{L^2(\mathbb{R}^n\times\mathbb{R}^m)}&=\left(\int_{\mathbb{R}^m}\|(\sum_{I_2}\sum_{J_2: J_2^{(l)}=I_2}M_1(\pair{f}{u_{J_2}}_2)^2\frac{\chi_{I_2}}{|I_2|})^{1/2}\|_{L^2(\mathbb{R}^n)}^2\,dx_2\right)^{1/2}\\
&\lesssim \left(\int_{\mathbb{R}^m}\|(\sum_{I_2}\sum_{J_2: J_2^{(l)}=I_2}|\pair{f}{u_{J_2}}_2|^2\frac{\chi_{I_2}}{|I_2|})^{1/2}\|_{L^2(\mathbb{R}^n)}^2\,dx_2\right)^{1/2}\\
&\lesssim\left(\int_{\mathbb{R}^m}\|f(\cdot,x_2)\|_{L^2(\mathbb{R}^n)}^2\,dx_2\right)^{1/2}=\|f\|_{L^2(\mathbb{R}^n\times\mathbb{R}^m)},
\end{split}
\]
where once again the last inequality is due to the same argument in the last part of the proof of Lemma \ref{paraB1}, thus the proof is complete.

\end{proof}

\begin{lem}\label{biparaPP}
Given $b,a\in BMO_{prod}(\mathbb{R}^n\times\mathbb{R}^m)$, define
\[
PP(b,a,f):=\sum_{I_1,I_2}\pair{b}{h_{I_1}\otimes u_{I_2}}\pair{f}{h_{I_1}\otimes u_{I_2}}|I_1|^{-1}|I_2|^{-1}\sum_{J_1: J_1\subsetneq I_1}\sum_{J_2: J_2\subsetneq I_2}\pair{a}{h_{J_1}\otimes u_{J_2}}h_{J_1}\otimes u_{J_2},
\]
and let $PP_1$ be its partial adjoint in the first variable with $b,a$ fixed. Then,
\begin{equation}\label{PP}
\|PP(b,a,f)\|_{L^2}\lesssim\|b\|_{BMO_{prod}}\|a\|_{BMO_{prod}}\|f\|_{L^2},
\end{equation}
\begin{equation}\label{PP1}
\|PP_1(b,a,f)\|_{L^2}\lesssim \|b\|_{BMO_{prod}}\|a\|_{BMO_{prod}}\|f\|_{L^2}.
\end{equation}
\end{lem}

Recall that for a bi-parameter singular integral $T$, its partial adjoint $T_1$ is defined via
\[
\pair{T(f_1\otimes f_2)}{g_1\otimes g_2}=\pair{T_1(g_1\otimes f_2)}{f_1\otimes g_2}.
\]
It is known that the $L^2$ boundedness of $T$ does not imply the $L^2$ boundedness of $T_1$ (see \cite{Jo} or \cite{MO} for a detailed discussion and counterexamples). Hence, in the following, we need to prove the boundedness of $PP$ and $PP_1$ separately. 

\begin{proof}
We first note that the proof of $PP$ is essentially the same as Lemma \ref{oneparaP}. In the bi-parameter setting, one needs to use the double square function $SS$ to characterize product $H^1$ and the strong maximal function $M_S$ as majorization. The key observation is that there holds the following bi-parameter John-Nirenberg inequality (see \cite{CF}):
\[
\|(\sum_{R\subset\Omega}|\pair{a}{h_R}|^2\frac{\chi_R}{|R|})^{1/2}\|_{L^p}\leq\|a\|_{BMO_{prod}}|\Omega|^{1/p},\quad 1<p<\infty,
\]
where $\Omega$ is any open set in $\mathbb{R}^n\times\mathbb{R}^m$ of finite measure, and $R$ denotes dyadic rectangles. It thus easy to verify that a same argument as in Lemma \ref{oneparaP} implies (\ref{PP}).

The estimate of (\ref{PP1}) involves the hybrid maximal-square functions, which we have seen in the proof of Lemma \ref{bipara}. To be specific, let $g\in L^2$ be a normalized test function,
\[
\begin{split}
&\pair{PP_1(b,a,f)}{g}\\
=&\pair{b}{\sum_{I_1,I_2}|I_1|^{-1}|I_2|^{-1}h_{I_1}\otimes u_{I_2}\sum_{J_1: J_1\subsetneq I_1}\sum_{J_2: J_2\subsetneq I_2}\pair{a}{h_{J_1}\otimes u_{J_2}}\pair{f}{h_{J_1}\otimes u_{I_2}}\pair{g}{h_{I_1}\otimes u_{J_2}}}.
\end{split}
\]

Note that by bi-parameter John-Nirenberg inequality,
\[
\begin{split}
&|\sum_{J_1: J_1\subsetneq I_1}\sum_{J_2: J_2\subsetneq I_2}\pair{a}{h_{J_1}\otimes u_{J_2}}\pair{f}{h_{J_1}\otimes u_{I_2}}\pair{g}{h_{I_1}\otimes u_{J_2}}|\\
=&|\sum_{J_1: J_1\subsetneq I_1}\sum_{J_2: J_2\subsetneq I_2}\pair{a}{h_{J_1}\otimes u_{J_2}}\pair{\pair{f}{u_{I_2}}_2\otimes\pair{g}{h_{I_1}}_1}{h_{J_1}\otimes u_{J_2}}|\\
\leq&\|a\|_{BMO_{prod}}|I_1||I_2|(\ave{|\pair{f}{u_{I_2}}_2|^p}_{I_1})^{1/p}(\ave{|\pair{g}{h_{I_1}}_1|^p}_{I_2})^{1/p},
\end{split}
\]
for some $1<p<2$. Hence,
\[
\begin{split}
&SS(\sum_{I_1,I_2}|I_1|^{-1}|I_2|^{-1}h_{I_1}\otimes u_{I_2}\sum_{J_1: J_1\subsetneq I_1}\sum_{J_2: J_2\subsetneq I_2}\pair{a}{h_{J_1}\otimes u_{J_2}}\pair{f}{h_{J_1}\otimes u_{I_2}}\pair{g}{h_{I_1}\otimes u_{J_2}})\\
\leq&\|a\|_{BMO_{prod}}\big(\sum_{I_1,I_2}(\ave{|\pair{f}{u_{I_2}}_2|^p}_{I_1})^{2/p}(\ave{|\pair{g}{h_{I_1}}_1|^p}_{I_2})^{2/p}\frac{\chi_{I_1}\otimes\chi_{I_2}}{|I_1||I_2|}\big)^{1/2}\\
\leq&\|a\|_{BMO_{prod}}\big(\sum_{I_2}M_1(|\pair{f}{u_{I_2}}_2|^p)^{2/p}\frac{\chi_{I_2}}{|I_2|}\big)^{1/2}\big(\sum_{I_1}M_2(|\pair{g}{h_{I_1}}_1|^p)^{2/p}\frac{\chi_{I_1}}{|I_1|}\big)^{1/2}.
\end{split}
\]

The two terms on the last line above can be viewed as generalized hybrid maximal-square functions, whose boundedness is easy to obtain. For example,
\[
\begin{split}
&\|\big(\sum_{I_2}M_1(|\pair{f}{u_{I_2}}_2|^p)^{2/p}\frac{\chi_{I_2}}{|I_2|}\big)^{1/2}\|_{L^2}\\
=&\left(\int_{\mathbb{R}^n}\|\big(\sum_{I_2}M_1(|\pair{f}{u_{I_2}}_2|^p)^{2/p}\frac{\chi_{I_2}}{|I_2|}\big)^{1/2}\|_{L^2(\mathbb{R}^m)}^2\,dx_1\right)^{1/2}\\
=&\left(\int_{\mathbb{R}^n}\sum_{I_2}M_1(|\pair{f}{u_{I_2}}_2|^p)^{2/p}\,dx_1\right)^{1/2}
\lesssim \left(\sum_{I_2}\int_{\mathbb{R}^n}|\pair{f}{u_{I_2}}_2|^2\,dx_1\right)^{1/2}=\|f\|_{L^2}.
\end{split}
\]

Therefore, $\|PP_1(b,a,f)\|_{L^2}\lesssim\|b\|_{BMO_{prod}}\|a\|_{BMO_{prod}}\|f\|_{L^2}$.
\end{proof}

In addition to the above two types of operators, in the bi-parameter setting, a new type of operator that mixes the paraproduct and $P$ arise naturally in our argument. We show that they have the following uniform BMO estimates.

\begin{lem}\label{biparaBP}
Given $b\in BMO_{prod}(\mathbb{R}^n\times\mathbb{R}^m)$, $a^1\in BMO(\mathbb{R}^n)$, and $a^2\in BMO(\mathbb{R}^m)$. For integers $k,l\geq 0$, define
\[
BP_k(b,a^2,f):=\sum_{I_1,I_2}\beta_{I_1}\pair{b}{h_{I_1^{(k)}}\otimes u_{I_2}}\pair{f}{h_{I_1}^{\epsilon_1}\otimes u_{I_2}}|I_1^{(k)}|^{-1/2}|I_2|^{-1}h_{I_1}^{\epsilon_1'}\sum_{J_2: J_2\subsetneq I_2}\pair{a^2}{u_{J_2}}_2u_{J_2},
\]
\[
PB_l(b,a^1,f):=\sum_{I_1,I_2}\beta_{I_2}\pair{b}{h_{I_1}\otimes u_{I_2^{(l)}}}\pair{f}{h_{I_1}\otimes u_{I_2}^{\epsilon_2}}|I_1|^{-1}|I_2^{(l)}|^{-1/2}h_{I_2}^{\epsilon_2'}\sum_{J_1: J_1\subsetneq I_1}\pair{a^1}{h_{J_1}}_1h_{J_1},
\]
where $\beta_{I_1}, \beta_{I_2}$ are sequences satisfying $|\beta_{I_1}, \beta_{I_2}|\leq 1$. When $k>0$, all the Haar functions in the first variable are cancellative, while when $k=0$, there is at most one of $h_{I_1}^{\epsilon_1}, h_{I_1}^{\epsilon_1'}$ being noncancellative. The same assumption goes for the second variable. Then, there holds
\[
\|BP_k(b,a^2,f)\|_{L^2}\lesssim \|b\|_{BMO_{prod}}\|a^2\|_{BMO}\|f\|_{L^2},
\]
\[
\|PB_l(b,a^1,f)\|_{L^2}\lesssim \|b\|_{BMO_{prod}}\|a^1\|_{BMO}\|f\|_{L^2}.
\]
\end{lem}

\begin{proof}
By symmetry, it suffices to estimate $PB_l$. The strategy is similar as before: a square function argument encoding the product BMO estimate of $b$, combined with a John-Nirenberg inequality taking advantage of the BMO estimate of $a^1$. Note that the arguments slightly vary depending on whether noncancellative Haar functions appear. Taking $g$ such that $\|g\|_{L^2}\leq 1$,
\[
\begin{split}
&\pair{PB_l(b,a^1,f)}{g}\\
=&\pair{b}{\sum_{I_1,I_2}\pair{f}{h_{I_1}\otimes u_{I_2}^{\epsilon_2}}|I_1|^{-1}|I_2^{(l)}|^{-1/2}h_{I_1}\otimes u_{I_2^{(l)}}\sum_{J_1: J_1\subsetneq I_1}\pair{a^1}{h_{J_1}}_1\pair{g}{h_{J_1}\otimes u_{I_2}^{\epsilon_2'}}}.
\end{split}
\]

A similar application of John-Nirenberg inequality as before implies that
\begin{equation}\label{square}
\begin{split}
&SS(\sum_{I_1,I_2}\pair{f}{h_{I_1}\otimes u_{I_2}^{\epsilon_2}}|I_1|^{-1}|I_2^{(l)}|^{-1/2}h_{I_1}\otimes u_{I_2^{(l)}}\sum_{J_1: J_1\subsetneq I_1}\pair{a^1}{h_{J_1}}_1\pair{g}{h_{J_1}\otimes u_{I_2}^{\epsilon_2'}})\\
\leq&\|a^1\|_{BMO}\left(\sum_{I_1,J_2}\big(\sum_{I_2\subset J_2}^{(l)}\pair{f}{h_{I_1}\otimes u_{I_2}^{\epsilon_2}}(\ave{|\pair{g}{u_{I_2}^{\epsilon_2'}}_2|^p}_{I_1})^{1/p}\big)^2\frac{\chi_{I_1}\otimes\chi_{J_2}}{|I_1||J_2|^2}\right)^{1/2}
\end{split}
\end{equation}

(a) Case $l>0$.

In this case, all the Haar functions that appear are cancellative, hence by omitting the dependence on $\epsilon_2,\epsilon_2'$ and applying Cauchy-Schwarz inequality, there holds
\[
\begin{split}
(\ref{square})&\leq\|a^1\|_{BMO}\left(\sum_{I_1,J_2}\big(\sum_{I_2\subset J_2}^{(l)}|\pair{f}{h_{I_1}\otimes u_{I_2}}|^2\big)\big(\sum_{I_2\subset J_2}^{(l)}(\ave{|\pair{g}{u_{I_2}}_2|^p}_{I_1})^{2/p}\big)\frac{\chi_{I_1}\otimes\chi_{J_2}}{|I_1||J_2|^2}\right)^{1/2}\\
&\leq\|a^1\|_{BMO}\left(\sum_{J_2}\big(\sum_{I_1}\sum_{I_2\subset J_2}^{(l)}|\pair{f}{h_{I_1}\otimes u_{I_2}}|^2\frac{\chi_{I_1}}{|I_1|}\big)\big(\sum_{I_2\subset J_2}^{(l)}M_1(|\pair{g}{u_{I_2}}_2|^p)^{2/p}\big)\frac{\chi_{J_2}}{|J_2|^2}\right)^{1/2},
\end{split}
\]
which by $\|\cdot\|_{\ell^2}\leq\|\cdot\|_{\ell^1}$ and another use of Cauchy-Schwarz is bounded by
\[
\|a^1\|_{BMO}\left(\sum_{I_1}\sum_{J_2}\sum_{I_2\subset J_2}^{(l)}|\pair{f}{h_{I_1}\otimes u_{I_2}}|^2\frac{\chi_{I_1}\otimes\chi_{J_2}}{|I_1||J_2|}\right)^{1/2}\left(\sum_{J_2}\sum_{I_2\subset J_2}^{(l)}M_1(|\pair{g}{u_{I_2}}_2|^p)^{2/p}\frac{\chi_{J_2}}{|J_2|}\right)^{1/2}.
\]

Therefore, a similar double square function and hybrid maximal-square function argument as in Lemma \ref{bipara} and Lemma \ref{biparaPP} implies that
\[
\|(\ref{square})\|_{L^1}\lesssim\|a^1\|_{BMO}\|f\|_{L^2}\|g\|_{L^2}.
\]

(b) Case $l=0$ and $\epsilon_2=\vec{1}$.

In this case, 
\[
\begin{split}
(\ref{square})&=\|a^1\|_{BMO}\left(\sum_{I_1,I_2}\big(\ave{\pair{f}{h_{I_1}}_1}_{I_2}\big)\big(\ave{|\pair{g}{u_{I_2}}_2|^p}_{I_1}\big)^{2/p}\frac{\chi_{I_1}\otimes\chi_{I_2}}{|I_1||I_2|}\right)^{1/2}\\
&\leq \left(\sum_{I_1}M_2(\pair{f}{h_{I_1}}_1)^2\frac{\chi_{I_1}}{|I_1|}\right)^{1/2}\left(\sum_{I_2}M_1(|\pair{g}{u_{I_2}}_2|^p)^{2/p}\frac{\chi_{I_2}}{|I_2|}\right)^{1/2},
\end{split}
\]
which shows that $\|(\ref{square})\|_{L^1}\lesssim\|a^1\|_{BMO}\|f\|_{L^2}\|g\|_{L^2}$.

(c) Case $l=0$ and $\epsilon_2'=\vec{1}$.

This last case can be dealt with similarly by noticing that
\[
\begin{split}
(\ref{square})&=\|a^1\|_{BMO}\left(\sum_{I_1,I_2}|\pair{f}{h_{I_1}\otimes u_{I_2}}|^2(\ave{|\ave{g}_{I_2}|^p}_{I_1})^{2/p}\frac{\chi_{I_1}\otimes\chi_{I_2}}{|I_1||I_2|}\right)^{1/2}\\
&\leq\|a^1\|_{BMO}\left(M_1(|M_2(g)|^p)\right)^{1/p}SS(f).
\end{split}
\]
The boundedness of $M_1$ and $M_2$ in each variable implies that
\[
\|\left(M_1(|M_2(g)|^p)\right)^{1/p}\|_{L^2}\lesssim \|g\|_{L^2}.
\]

To conclude, we've demonstrated in each case that
\[
\|PB_l(b,a^1,f)\|_{L^2}\lesssim\|b\|_{BMO_{prod}}\|a^1\|_{BMO}\|f\|_{L^2},
\]
which completes the proof.
\end{proof}

Now let's proceed with the proof of Theorem \ref{main}. Using Theorem \ref{repre} twice for both variables we have
\[
\begin{split}
&[[b,T_1],T_2]f\\
&=c\|T_1\|_{CZ}\|T_2\|_{CZ}\mathbb{E}_{\omega_1}\mathbb{E}_{\omega_2}\sum_{i_1,j_1=0}^\infty\sum_{i_2,j_2=0}^\infty 2^{-\max(i_1,j_1)\frac{\delta}{2}}2^{-\max(i_2,j_2)\frac{\delta}{2}}[[b,S^{i_1j_1}_{\omega_1}],S^{i_2j_2}_{\omega_2}]f.
\end{split}
\]
Since our estimate in the following doesn't depend on the parameters $\omega_1,\omega_2$ explicitly, we will omit them in the notation. Our goal is to prove that
\[
\begin{split}
&\|[[b,S^{i_1j_1}_1],S^{i_2j_2}_2]f\|_{L^2(\R^n\times\R^m)}\\
&\lesssim (1+\max(i_1,j_1))(1+\max(i_2,j_2))\|b\|_{BMO_{prod}(\R^n\times\R^m)}\|f\|_{L^2(\R^n\times\R^m)},
\end{split}
\]
which can be achieved by showing that any $[[b,S^{i_1j_1}_1],S^{i_2j_2}_2]f$ can be represented as a finite linear combination of the following terms and their adjoints (which is understood as the adjoint operator with $b, a^i$ fixed):
\[
B_{k,l}(b, S^{i_1j_1}_1S^{i_2j_2}_2f),\quad S^{i_1j_1}_1(B_{k,l}(b,S^{i_2j_2}_2f)),
\]
\[
BP_k(b,a^2,S^{i_1j_1}_1f),\quad PB_l(b,a^1,S^{i_2j_2}_2f), 
\]
\[
PP(b,a^1\otimes a^2,f),\quad PP_1(b, a^1\otimes a^2,f).
\]
where $k,l\geq 0$, and $a^i$ is the BMO symbol of the dyadic shift $S^{00}$ if it appears in the $i$-th variable. The total number of terms in the representation is no greater than $C(1+\max(i_1,j_1))(1+\max(i_2,j_2))$ for some universal constant $C$. Note that for $a^1\in BMO(\mathbb{R}^n)$ and $a^2\in BMO(\mathbb{R}^m)$, there holds $a^1\otimes a^2\in BMO_{prod}(\mathbb{R}^n\times\mathbb{R}^m)$. Hence, implied by Theorem \ref{para}, Lemma \ref{bipara}, Lemma \ref{biparaPP} and Lemma \ref{biparaBP}, the $L^2$ norm of all of the terms above are uniformly bounded, independent of $k,l$ in particular.

To derive the desired representation, we argue by an iteration of Theorem \ref{decomp}. 

\subsection{Cancellative dyadic shifts $S^{i_1j_1}_1$ and $S^{i_2j_2}_2$}\label{cancan}

In the case when both $S^{i_1j_1}_1$ and $S^{i_2j_2}_2$ are cancellative, only operators $B_{k,l}$ need to be involved. In order to make the notations clear, in the following, we will use $B_k^\tau$ to denote the one-parameter paraproducts that appeared in the previous section for the $\tau$-th variable, where $k\geq 0$ and $\tau=1,2$. Calculation shows that
\[
[[b, S_1^{i_1j_1}], S_2^{i_2j_2}]f=\sum_{I_1,J_1}\sum_{I_2,J_2}\pair{b}{h_{I_1}\otimes u_{I_2}}\pair{f}{h_{J_1}\otimes u_{J_2}}[h_{I_1},S_1^{i_1j_1}]h_{J_1}\otimes[u_{I_2},S_2^{i_2j_2}]u_{J_2},
\]
which by iteration equals
\[
\begin{split}
\sum_{I_2,J_2}\Big(&\sum_{t_1\in\Lambda_1}B_{k,t_1}^1(\pair{b}{u_{I_2}}_2,S^{i_1j_1}_1(\pair{f}{u_{J_2}}_2))\\
&+\sum_{t_2\in\Lambda_2}S^{i_1j_1}_1(B_{k,t_2}^1(\pair{b}{u_{I_2}}_2,\pair{f}{u_{J_2}}_2))\Big)\otimes\left([u_{I_2},S^{i_2j_2}_2]u_{J_2}\right),
\end{split}
\]
where $B_{k,t_i}^1$ are paraproducts of type $B_k^1$ in the first variable, and for each $t_i$, $k$ is an arbitrary nonnegative integer. Note that in the first parentheses we have a finite linear combination of terms that have already been studied in the previous section, and all of the index set $\Lambda_i$ satisfy $|\Lambda_i|\leq C(1+\max(i_1,j_1))$, $i=1,2$. Since the terms inside the first parentheses can be treated similarly, let's study one of the terms $B^1_{k,t_1}$ as an example. We will also omit the subscript $t_1$ as the choice is arbitrary. Then, the sum corresponding to $B_k^1$ is equal to
\[
\begin{split}
&\sum_{I_2,J_2}B_k^1(\pair{b}{u_{I_2}}_2,S_1^{i_1j_1}(\pair{f}{u_{J_2}}_2))\otimes\left([u_{I_2},S^{i_2j_2}_2]u_{J_2}\right)\\
&=\sum_{I_2,J_2}\sum_{I_1}\beta_{I_1}\pair{b}{h_{I_1^{(k)}}\otimes u_{I_2}}\pair{S_1^{i_1j_1}(f)}{h_{I_1}^{\epsilon_1}\otimes u_{J_2}}h_{I_1}^{\epsilon_1'}|I_1^{k}|^{-1/2}\otimes \left([u_{I_2},S^{i_2j_2}_2]u_{J_2}\right)\\
&=\sum_{I_1}\beta_{I_1}h_{I_1}^{\epsilon_1'}|I_1^{(k)}|^{-1/2}\otimes \left([\pair{b}{h_{I_1^{(k)}}}_1,S^{i_2j_2}_2]\pair{S_1^{i_1j_1}f}{h_{I_1}^{\epsilon_1}}_1\right)\\
&=\sum_{I_1}\beta_{I_1}h_{I_1}^{\epsilon_1'}|I_1^{(k)}|^{-1/2}\otimes \Big(\sum_{s_1\in\Gamma_1}B_{l,s_1}^2(\pair{b}{h_{I_1^{(k)}}}_1,S^{i_2j_2}_2(\pair{S_1^{i_1j_1}(f)}{h_{I_1}^{\epsilon_1}}_1))\\
&\qquad +\sum_{s_2\in\Gamma_2}S^{i_2j_2}_2(B_{l,s_2}^2(\pair{b}{h_{I_1^{(k)}}}_1,\pair{S_1^{i_1j_1}(f)}{h_{I_1}^{\epsilon_1}}_1))\Big),
\end{split}
\]
where $B_{l,s_i}^2$ are paraproducts of type $B_l^2$ in the second variable, and all the index sets $\Gamma_i$ satisfy $|\Gamma_i|\leq C(1+\max(i_2,j_2))$, $i=1,2$. Again, since all the terms in the parentheses are similar, we only consider one of $B^2_{l,s_2}$ and omit the subscript $s_2$. This is a mixed case, and all the other combinations follow similarly. Thus, noticing that
\[
\begin{split}
&\sum_{I_1}\beta_{I_1}h_{I_1}^{\epsilon_1'}|I_1^{(k)}|^{-1/2}\otimes S_2^{i_2j_2}(B_l^2(\pair{b}{h_{I_1^{(k)}}}_1,\pair{S_1^{i_1j_1}(f)}{h_{I_1}^{\epsilon_1}}_1))\\
&=S_2^{i_2j_2}\left(\sum_{I_1,I_2}\beta_{I_1}\beta_{I_2}\pair{b}{h_{I_1^{(k)}}\otimes u_{I_2^{(l)}}}\pair{S_1^{i_1j_1}f}{h_{I_1}^{\epsilon_1}\otimes u_{I_2}^{\epsilon_2}}h_{I_1}^{\epsilon_1'}\otimes u_{I_2}^{\epsilon_2'}|I_1^{(k)}|^{-\frac{1}{2}}|I_2^{(l)}|^{-\frac{1}{2}}\right)
\end{split}
\]
is exactly $S_2^{i_2j_2}(B_{k,l}(b,S_1^{i_1j_1}f))$, where $B_{k,l}$ is the bi-parameter paraproduct we've studied in Lemma \ref{bipara}, and the only case involving noncancellative Haar functions is when the corresponding $k$ or $l$ is $0$. We therefore obtain the desired representation of this term. All the other terms can be treated similarly, by noticing that paraproducts $B_{k,l}$ can be obtained by combining $B_k^1$ and $B_l^2$ through the same process described above. And it is easily seen that the total number of terms is bounded by $(1+\max(i_1,j_1))(1+\max(i_2,j_2))$ up to a dimensional constant. 

\subsection{Cancellative dyadic shift $S^{i_1j_1}_1$ and noncancellative dyadic shift $S_2^{00}$}\label{cannon}

We assume that $S^{00}_2f=\sum_{I_2}\pair{a^2}{u_{I_2}}_2|I_2|^{-1/2}\pair{f}{u_{I_2}^1}_2u_{I_2}$. Following from Theorem \ref{decomp}, in the first variable, the commutator can be represented as a linear combination of paraproducts, i.e.
\[
\begin{split}
[[b,S^{i_1j_1}_1],S^{00}_2]f&=\sum_{I_1\subset J_1^{(i_1)}}\sum_{I_2\subset J_2}\pair{b}{h_{I_1}\otimes u_{I_2}}\pair{f}{h_{J_1}\otimes u_{J_2}}[h_{I_1},S^{i_1j_1}_1]h_{J_1}\otimes [u_{I_2},S^{00}_2]u_{J_2}\\
&=\sum_{I_2\subset J_2}\big(\sum_{t_1\in\Lambda_1}B_{k,t_1}^1(\pair{b}{u_{I_2}}_2,S^{i_1j_1}_1(\pair{f}{u_{J_2}}_2))\\
&\qquad+\sum_{t_2\in\Lambda_2}S^{i_1j_1}_1(B_{k,t_2}^1(\pair{b}{u_{I_2}}_2,\pair{f}{u_{J_2}}_2))\Big)\otimes\left([u_{I_2},S^{00}_2]u_{J_2}\right).
\end{split}
\]

Recall that by Theorem \ref{decomp}, in the one-parameter setting, the noncancellative dyadic shift $S^{00}$ can be represented as a finite linear combination of paraproducts (corresponding to the sum over $I\subsetneq J$ and the second term in the sum over $I=J$) and operator $P$ (corresponding to the first term in the sum over $I=J$). Hence, 
\[
\begin{split}
&\sum_{I_2\subset J_2}B_{k,t_1}^1(\pair{b}{u_{I_2}}_2,S^{i_1j_1}_1(\pair{f}{u_{J_2}}_2))\otimes [u_{I_2},S^{00}_2]u_{J_2}\\
=&\sum_{I_1}\beta_{I_1}h_{I_1}^{\epsilon_1'}|I_1^{(k)}|^{-1/2}\otimes\big([\pair{b}{h_{I_1^{(k)}}}_1,S^{00}_2]\pair{S^{i_1j_1}_1f}{h_{I_1}^{\epsilon_1}}_1\big)\\
=&\sum_{I_1}\beta_{I_1}h_{I_1}^{\epsilon_1'}|I_1^{(k)}|^{-1/2}\otimes\big(\sum_{s_1\in\Gamma_1}B_{0,s_1}^2(\pair{b}{h_{I_1^{(k)}}}_1,S^{00}_2(\pair{S^{i_1j_1}_1f}{h_{I_1}^{\epsilon_1}}_1))\\
&\qquad +\sum_{s_2\in\Gamma_2}S^{00}_2(B_{0,s_2}^2(\pair{b}{h_{I_1^{(k)}}}_1,\pair{S^{i_1j_1}_1f}{h_{I_1}^{\epsilon_1}}_1))+P(\pair{b}{h_{I_1^{(k)}}}_1,a^2,\pair{S^{i_1j_1}_1f}{h_{I_1}^{\epsilon_1}}_1)\\
=&\big(\sum_{s_1\in\Gamma_1}B_{k,0,s_1}(b,S^{i_1j_1}_1S^{00}_2f)\big)+\big(\sum_{s_2\in\Gamma_2}S^{00}_2(B_{k,0,s_2}(b,S^{i_1j_1}_1f))\big)+BP_k(b,a^2,S^{i_1j_1}_1f).
\end{split}
\]

Similarly, the other term can be treated exactly the same:
\[
\begin{split}
&\sum_{I_2\subset J_2}S^{i_1j_1}_1(B_{k,t_2}^1(\pair{b}{u_{I_2}}_2,\pair{f}{u_{J_2}}_2))\otimes [u_{I_2},S^{00}_2]u_{J_2}\\
=&\big(\sum_{s_1\in\Gamma_1}S^{i_1j_1}_1(B_{k,0,s_1}(b,S^{00}_2f))\big)+\big(\sum_{s_2\in\Gamma_2}S^{i_1j_1}_1S^{00}_2(B_{k,0,s_2}(b,f))\big)+S^{i_1j_1}_1(BP_k(b,a^2,f)).
\end{split}
\]

The desired representation is hence obtained. Note that by symmetry and duality, this implies the boundedness of other types of the mixed cases as well.

\subsection{Noncancellative dyadic shfits $S_1^{00}$ and $S_2^{00}$}

Write
\[
[[b,S^{00}_1],S^{00}_2]f=\sum_{I_1\subset J_1}\sum_{I_2\subset J_2}\pair{b}{h_{I_1}\otimes u_{I_2}}\pair{f}{h_{J_1}\otimes u_{J_2}}[h_{I_1},S^{00}_1]u_{J_1}\otimes [h_{I_2},S^{00}_2]u_{J_2}.
\]

First, we deal with the case when both $S^{00}_1$ and $S^{00}_2$ are of the same type, for instance,
\[
S^{00}_1f:=\sum_{I_1}\pair{a^1}{h_{I_1}}_1|I_1|^{-1/2}\pair{f}{h_{I_1}^1}h_{I_1},\quad S^{00}_2f:=\sum_{I_2}\pair{a^2}{u_{I_2}}_2|I_2|^{-1/2}\pair{f}{u_{I_2}^1}_2u_{I_2}.
\]

Observe that compared with section \ref{cancan} and \ref{cannon}, after decomposing the commutator in each variable into paraproducts and operator $P$, the only new case that arises here is the "tensor product" of operator $P$ in both variables, which is equal to
\[
\begin{split}
&\sum_{I_1,I_2}\pair{b}{h_{I_1}\otimes u_{I_2}}\pair{f}{h_{I_1}\otimes u_{I_2}}|I_1|^{-1}|I_2|^{-1}\sum_{J_1: J_1\subsetneq I_1}\sum_{J_2: J_2\subsetneq I_2}\pair{a^1\otimes a^2}{h_{J_1}\otimes u_{J_2}}h_{J_1}\otimes u_{J_2}\\
=&PP(b,a^1\otimes a^2,f).
\end{split}
\]

Second, we discuss the case when $S^{00}_1$ and $S^{00}_2$ are of different types, for instance,
\[
S^{00}_1f:=\sum_{I_1}\pair{a^1}{h_{I_1}}_1|I_1|^{-1/2}\pair{f}{h_{I_1}}h_{I_1}^1,\quad S^{00}_2f:=\sum_{I_2}\pair{a^2}{u_{I_2}}_2|I_2|^{-1/2}\pair{f}{u_{I_2}^1}_2u_{I_2}.
\]

It is implied by Theorem \ref{decomp} that in the first variable, the commutator is a linear combination of paraproducts and operator $P^*$. Therefore, the only new case that arises here in the representation is $P^*$ in the first variable mixed with $P$ in the second variable, which is
\[
\begin{split}
&\sum_{I_1,I_2}\pair{b}{h_{I_1}\otimes u_{I_2}}|I_1|^{-1}|I_2|^{-1}\sum_{J_1: J_1\subsetneq I_1}\sum_{J_2: J_2\subsetneq I_2}\pair{a^1\otimes a^2}{h_{J_1}\otimes u_{J_2}}\pair{f}{h_{J_1}\otimes u_{I_2}}h_{I_1}\otimes u_{J_2}\\
=&PP_1(b,a^1\otimes a^2,f).
\end{split}
\]

Hence the main theorem in the bi-parameter setting is proved. As a final remark, the proof in the multi-parameter setting proceeds exactly the same as this one. Clearly, in the desired representation of commutators with dyadic shifts, one needs to involve a larger number of basic operators which mix together $B_k$ and $P$ in each variable, but the uniform boundedness of such operators can all be obtained similarly as in Lemma \ref{bipara}, \ref{biparaPP} and \ref{biparaBP}.

\end{document}